\documentclass[12pt]{article}
\usepackage{amsfonts}
\usepackage{amssymb}
\usepackage{mathrsfs}
\usepackage{srcltx}
\usepackage{amsmath}
\usepackage{amsthm}
\usepackage{amstext}
\usepackage{amsopn}
\usepackage{graphicx}

\newtheorem{theorem}{Theorem}
\newtheorem{lemma}{Lemma}[section]

\theoremstyle{definition}

\theoremstyle{remark}

\numberwithin{equation}{section}

\newcommand{\ba}{\begin{array}}
\newcommand{\ea}{\end{array}}
\newcommand{\f}{\frac}

\newcommand{\Om}{\Omega}

\newcommand{\la}{\lambda}
\newcommand{\R}{{\mathbf R}}
\newcommand{\ds}{\displaystyle}

\begin{document}
\date{}
\title{ \bf\large{Stability of Synchronized Steady State Solution of Diffusive Lotka-Volterra Predator-Prey Model}\footnote{Partially supported by National Natural Science Foundation of China (Grant Nos. 11801338, 11671239, 11571209), Science Council of Shanxi Province (Grant Nos. 201801D211001,  201801D221012), and  US-NSF grant DMS-1853598.}}
\author{Yongyan Huang\textsuperscript{1}\ \ Fuyi Li\textsuperscript{1}\ \ Junping Shi\textsuperscript{2}\footnote{Corresponding Author, Email: jxshix@wm.edu}
 \\
{\small \textsuperscript{1}School of Mathematical Sciences, Shanxi University, \hfill{\ }}\\
\ \ {\small Taiyuan, Shanxi, 030006, P.R. China. \hfill {\ }}\\
{\small \textsuperscript{2} Department of Mathematics, William \& Mary,\hfill{\ }}\\
\ \ {\small Williamsburg, Virginia, 23187-8795, USA.\hfill {\ }} }
\maketitle
\begin{abstract}
{We study the reaction-diffusion Lotka-Volterra predator-prey model with Dirichlet boundary condition. In the case of equal diffusion rates and equal growth rates, the synchronized steady state solution is proved to be locally asymptotically stable.}

\end{abstract}

\section {Introduction and Main Result}
In this paper, we  show the local asymptotic stability of synchronized solution of the following reaction-diffusion Lotka-Volterra predator-prey model
\begin{equation}\label{1.1}
 \begin{cases}
   \ds \Delta u+u[a-u-bv]=0, & \;\; x\in \Omega,\\
   \ds \Delta v+v[a-v+cu]=0, & \;\; x\in \Omega, \\
   \ds u=v=0, & \;\; x\in\partial\Omega.
 \end{cases}
\end{equation}
Here the functions $u(x)$ and $v(x)$ represent the population densities of prey and predator for $x\in \Om$ respectively; $\Om$ is a bounded connect open subset of $\R^n$ ($n\ge 1$) with a smooth enough boundary $\partial \Omega$. The diffusion coefficients and self-regulation of each species all equal to $1$,  the growth rates of the two species are the same constant $a>0$, and the predation rates $b>0$, $c>0$.

Our main result is stated as
\begin{theorem}\label{thm:1} Assume $0<b<1$ and $c>0$.
Let $\lambda_1$ be the principal eigenvalue of $-\Delta$ in $H_0^1(\Omega)$.
When  $a>\lambda_1$,  \eqref{1.1} has a positive solution
\begin{equation}\label{1.3}
(u_a,v_a)=(\frac{1-b}{1+bc}\theta_a,\frac{1+c}{1+bc}\theta_a)
\end{equation}
where $\theta_a$ is the unique positive solution of
\begin{equation}\label{1.3a}
  \begin{cases}
   \ds \Delta\theta+\theta(a-\theta)=0, & \;\; x\in \Omega,\\
   \ds \theta=0, & \;\; x\in \partial\Omega,
  \end{cases}
 \end{equation}
and $(u_a,v_a)$ is locally asymptotically stable.
\end{theorem}

Here the local asymptotic stability of  $(u_a,v_a)$ is with respect to the reaction-diffusion Lotka-Volterra predator-prey model:
\begin{equation}\label{1.1a}
 \begin{cases}
   \ds u_t=\Delta u+u[a-u-bv], & \;\; x\in \Omega,\;t>0,\\
   \ds v_t=\Delta v+v[a-v+cu], & \;\; x\in \Omega,\;t>0,\\
   \ds u(x,t)=v(x,t)=0, & \;\; x\in\partial\Omega,\;t>0\\
   u(x,0)=u_0(x)\ge 0, \;\; v(x,0)=v_0(x)\ge 0, & \;\; x\in \Omega.
 \end{cases}
\end{equation}
The steady state solution $(u_a,v_a)$ is referred as a synchronized solution as $u_a/v_a$ is a constant.

The existence of a positive steady state solution of \eqref{1.1a} when the growth rates are different for the two species is well known, see for example \cite{BB,Dancer,K}. For the one-dimensional spatial domain $\Om=(0,L)$, the uniqueness of positive solution was proved in \cite{MR02}. Here we show that the synchronized positive steady state of \eqref{1.1} is locally asymptotically stable, but it is unclear whether the synchronized solution is the only positive solution of \eqref{1.1}. We are indebted to one of the referees to tell us that, with an extra condition $c\in (0,(1-b)/b)$, the synchronized solution $(u_a,v_a)$ is shown to be unique and globally asymptotically stable \cite[Lemma 3.6]{Cantrell1993}, and that result is based on upper-lower solution method which is different from the one here. So here we prove the local stability of $(u_a,v_a)$ without the extra condition, but the global stability is not known. When the growth rates of predator and prey are different, the uniqueness and global stability of positive steady state  remains an interesting open question. We remark that our result still holds when the constant $a$ is replaced by  a function  $a\in C^\alpha(\overline{\Omega})$ for $\alpha\in (0,1)$ and $a(x)>0$ for $x\in \overline{\Omega}$.

The dynamic behavior of reaction-diffusion Lotka-Volterra competition model with same growth (resource) function has been studied more extensively. In \cite{MR01}, it is shown that the positive synchronized steady state solution for Dirichlet boundary value problem is unique and globally asymptotically stable. The effect of spatial heterogeneity and global stability of steady state for Neumann boundary value problem with non-constant growth function are studied in \cite{He,MR07,Lou}. The long time behavior of competition model when $a\equiv1$ is discussed in \cite{MR06}.

\section {Proof}
In order to prove the main result, first we recall the following well-known result:
\begin{lemma}\label{lem:2.1}
If $a>\lambda_1$, then \eqref{1.3a} has  a unique positive solution $\theta_a(x)$.
\end{lemma}
This result was known in \cite{MR07}, see also \cite[Theorem 25.4]{MR08} or \cite[Proposition 3.3]{MR05}. To state the stability of $\theta_a$ with respect to the reaction diffusion equation:
\begin{equation}\label{1.3b}
  \begin{cases}
   \ds u_t=\Delta u+u(a-u), & \;\; x\in \Omega,\, t>0,\\
   \ds u(x,0)=0, & \;\; x\in \partial\Omega,\; t>0,\\
    u(x,0)=u_0(x)\ge 0, \;\; & \;\; x\in \Omega,
  \end{cases}
 \end{equation}
we recall some definitions and preliminary results.

Let $m(x) \in L^\infty(\Omega)$.
We define a linear operator
\begin{equation*}
 \ds L(m)\varphi=\Delta \varphi+m(x)\varphi, \;\;\; \varphi\in W^{2,p}_0(\Om),
\end{equation*}
where $p>n$. Then the eigenvalue problem
\begin{equation*}
 \begin{cases}
  \ds L(m)\varphi=-\lambda\varphi & \;\; x\in \Omega,\\
  \ds \varphi=0, &\;\; x\in\partial\Omega,
 \end{cases}
\end{equation*}
has a sequence of eigenvalues $\la_i(m)$ satisfying $\lambda_1(m)<\lambda_2(m)\leqslant\cdots\leqslant\lambda_i(m)\leqslant\lambda_{i+1}(m)\to \infty$, and
the corresponding eigenfunction is denoted by $\varphi_i(m)$ for $i=1,2,\cdots$. The set of all eigenfunctions $\{\varphi_i(m): i\in {\mathbb N}\}$ is an orthonormal basis of $L^2(\Om)$. In particular, $\la_1(m)$ is a simple eigenvalue of $-L(m)$ with a positive eigenfunction $\varphi_1(m)$, and all other eigenfunctions are sign-changing.
For convenience, the eigen-pair $(\lambda_i(0),\varphi_i(0))$ will be simply denoted by $(\lambda_i,\varphi_i)$ when $m\equiv0$.  From the variational characterization of $\la_i(m)$, if $m_1(x)\leqslant m_2(x)$ for $x\in \Om$, then
$\lambda_i(m_1)\geqslant\lambda_i(m_2)$ and $\lambda_i(m_1)>\lambda_i(m_2)$ when $m_1\not\equiv m_2$ for $i=1,2,\cdots$. 

From the monotonicity of eigenvalues with respect to the weight function, we have the following result implying the stability of $\theta_a$:

\begin{lemma}\label{lem:2.2}
Let $\theta_a$ be the unique positive solution of \eqref{1.3a}, and $s>1$. Consider the following eigenvalue problem
 \begin{equation}\label{2.1}
  \begin{cases}
   \ds \Delta\varphi+(a-s\theta_a)\varphi=-\lambda\varphi,&\;\; x\in \Omega,\\
   \ds \varphi=0, &\;\; x\in \partial \Omega.
  \end{cases}
 \end{equation}
Then the eigenvalues of \eqref{2.1} is a sequence satisfying $\lambda_{i,s}$ satisfying
$0<\lambda_{1,s}<\lambda_{2,s}\leqslant\cdots\leqslant\lambda_{i,s}\leqslant\lambda_{i+1,s}\to \infty$,  and the set of eigenfunctions $\{\varphi_{i,s}:i\in {\mathbb{N}}\}$ is an orthonormal basis of $L^2(\Om)$. In particular $\theta_a$ is a locally asymptotically stable steady state solution of \eqref{1.3b}.
\end{lemma}
\begin{proof}
The first part follows directly from discussion above and from the monotonicity of eigenvalues with respect to the weight function, we have
\begin{equation*}
    \la_{1,s}=\la_1(a-s\theta_a)>\la_1(a-\theta_a)=0,
\end{equation*}
since $\theta_a$ is the positive eigenfunction corresponding to $\la_1(a-\theta_a)$. The positive solution $\theta_a$ is locally asymptotically stable with respect to \eqref{1.3b} if all the eigenvalues of
\begin{equation}\label{2.1w}
 \begin{cases}
  \ds \Delta\varphi+(a-2\theta_a)\varphi=-\lambda\varphi, & \;\; x\in \Omega,\\
  \ds \varphi=0, &\;\; x\in\partial\Omega,
 \end{cases}
\end{equation}
are positive. So $\theta_a$ is  locally asymptotically stable as $\la_{1,2}>0$.
\end{proof}

Now we prove our main result.
\begin{proof}[Proof of Theorem \ref{thm:1}] It is easy to verify that the solution $(u_a,v_a)$ given in \eqref{1.3} is a positive solution of \eqref{1.1} if $b\in (0,1)$ and $c>0$.
We shall prove that $(u_a,v_a)$ is linearly stable (and also locally asymptotically stable). That is, all eigenvalues $\mu$ of
\begin{equation}\label{2.2}
 \begin{cases}
  \ds \Delta\phi+(a-2u_a-bv_a)\phi-bu_a\psi=-\mu\phi,&\;\; x\in \Omega,\\
  \ds \Delta\psi+(a-2v_a+cu_a)\psi+cv_a\phi=-\mu\psi,&\;\; x\in \Omega,\\
  \ds \phi=\psi=0,&\;\; x\in \partial\Omega,
 \end{cases}
\end{equation}
have positive real parts. After substituting using \eqref{1.3},
\eqref{2.2} can be rewritten as
\begin{equation}\label{2.3}
 \begin{cases}
  \ds \Delta\phi+(a-(\frac{1-b}{1+bc}+1)\theta_a)\phi-b\frac{1-b}{1+bc}\theta_a\psi=-\mu\phi,&\;\; x\in \Omega,\\
  \ds \Delta\psi+(a-(\frac{1+c}{1+bc}+1)\theta_a)\psi+c\frac{1+c}{1+bc}\theta_a\phi=-\mu\psi,&\;\; x\in \Omega,\\
  \ds \phi=\psi=0,&\;\; x\in \partial\Omega.
 \end{cases}
\end{equation}

We claim that  there exist  $A_i,B_i\in \R$ and $s>1$ such that
\begin{equation}\label{2.4}
\binom{\phi_i}{\psi_i}=\binom{A_i}{B_i}\varphi_i(a-s\theta_a)
\end{equation}
is an eigenfunction  of \eqref{2.3} for $i=1,2,\cdots$, where $\varphi_i(a-s\theta_a)$ is the $i$-th eigenfunction of \eqref{2.1}.
First we assume that $b\ne c/(2c+1)$.
In this case, substituting \eqref{2.4} into \eqref{2.3} and using \eqref{2.1}, we find that $z=A_i/B_i$ must satisfy a quadratic equation
\begin{equation}\label{2.6}
c(1+c)z^2-(b+c)z+b(1-b)=0.
\end{equation}
It is easy to verify that \eqref{2.6} has two distinctive positive real roots $z_1=b/c$ and $z_2=(1-b)/(1+c)$, and
\begin{equation}\label{pq}
    \Phi_{i,1}=\binom{\phi_{i,1}}{\psi_{i,1}}:=\binom{b}{c}\varphi_i(a-s\theta_a), \;\; \Phi_{i,2}=\binom{\phi_{i,2}}{\psi_{i,2}}:=\binom{1-b}{1+c}\varphi_i(a-s\theta_a)
\end{equation}
are two linearly independent eigenfunctions of \eqref{2.3} with $s=\ds\frac{2+c-b}{1+bc}>1$ as $b\in (0,1)$ and $c>0$.    The corresponding eigenvalue $\mu=\mu_i=\la_i(a-s_1\theta_a)>0$ from Lemma \ref{lem:2.2}. Moreover since $\{\varphi_{i,s}:i\in {\mathbb{N}}\}$ is an orthonormal basis of $L^2(\Om)$, the set
\begin{equation*}
    \left\{\binom{\phi_{i,1}}{\psi_{i,1}},\binom{\phi_{i,2}}{\psi_{i,2}}:i\in {\mathbb N}\right\}
\end{equation*}
is an orthogonal basis of $L^2(\Omega)\times L^2(\Omega)$. This in turn implies that the set of eigenvalues of \eqref{2.3} is exactly$\{\la_i(a-s_1\theta_a):i\in {\mathbb N}\}$. Secondly if $b= c/(2c+1)$, the above argument still shows that $\Phi_{i,1}$ and $ \Phi_{i,2}$ defined in \eqref{pq} are eigenfunctions of \eqref{2.3} with $s=2>1$ but $\Phi_{i,1}=k\Phi_{i,2}$ for some constant $k\ne 0$. In this case, if $(\phi,\psi)$ is an eigenfunction of \eqref{2.3} with eigenvalue $\mu$, then it is easy to verify that $\xi=(2c+1)\phi-\psi$ satisfies
\begin{equation}\label{2.1ww}
 \begin{cases}
  \ds \Delta\xi+(a-2\theta_a)\xi=-\mu\xi, & \;\; x\in \Omega,\\
  \ds \xi=0, &\;\; x\in\partial\Omega.
 \end{cases}
\end{equation}
That is, $\mu$ must be equal to some  $\la_i(a-2\theta_a)$, and that in turn implies that $(\phi,\psi)$ is a multiple of $\Phi_{i,1}$. So we can also conclude that the set of eigenvalues of \eqref{2.3} is exactly$\{\la_i(a-s_1\theta_a):i\in {\mathbb N}\}$.

 Therefore $(u_a,v_a)$ is linearly stable and also locally asymptotically stable since $\la_i(a-s_1\theta_a)>0$ for each $i\in {\mathbb N}$.
\end{proof}

\noindent{\bf Acknowledgment}: The authors thank two anonymous reviewers for their very helpful suggestions which greatly improve the initial draft.

\bibliographystyle{plain}
\bibliography{huangyongyan-2-bib}

\begin{thebibliography}{10}

\bibitem{MR08}
H.~Amann.
\newblock Fixed point equations and nonlinear eigenvalue problems in ordered
  {B}anach spaces.
\newblock {\em SIAM Rev.}, 18(4):620--709, 1976.

\bibitem{BB}
J.~Blat and K.~J. Brown.
\newblock Bifurcation of steady-state solutions in predator-prey and
  competition systems.
\newblock {\em Proc. Roy. Soc. Edinburgh Sect. A}, 97:21--34, 1984.

\bibitem{MR05}
R.~S. Cantrell and C.~Cosner.
\newblock {\em Spatial ecology via reaction-diffusion equations}.
\newblock Wiley Series in Mathematical and Computational Biology. John Wiley \&
  Sons, Ltd., Chichester, 2003.

\bibitem{Cantrell1993}
R.~S. Cantrell, C.~Cosner, and V.~Hutson.
\newblock Permanence in some diffusive {L}otka-{V}olterra models for three
  interacting species.
\newblock {\em Dynam. Systems Appl.}, 2(4):505--530, 1993.

\bibitem{MR01}
C.~Cosner and A.~C. Lazer.
\newblock Stable coexistence states in the {V}olterra-{L}otka competition model
  with diffusion.
\newblock {\em SIAM J. Appl. Math.}, 44(6):1112--1132, 1984.

\bibitem{Dancer}
E.~N. Dancer.
\newblock On positive solutions of some pairs of differential equations.
\newblock {\em Trans. Amer. Math. Soc.}, 284(2):729--743, 1984.

\bibitem{He}
X.~Q. He and W.-M. Ni.
\newblock Global dynamics of the {L}otka-{V}olterra competition-diffusion
  system: diffusion and spatial heterogeneity {I}.
\newblock {\em Comm. Pure Appl. Math.}, 69(5):981--1014, 2016.

\bibitem{K}
P.~Korman.
\newblock Dynamics of the {L}otka-{V}olterra systems with diffusion.
\newblock {\em Appl. Anal.}, 44(3-4):191--207, 1992.

\bibitem{MR07}
M.~A. Krasnosel'ski\u{\i}.
\newblock {\em Positive solutions of operator equations}.
\newblock P. Noordhoff Ltd. Groningen, 1964.

\bibitem{MR02}
J.~L\'{o}pez-G\'{o}mez and R.~Pardo.
\newblock Existence and uniqueness of coexistence states for the predator-prey
  model with diffusion: the scalar case.
\newblock {\em Differential Integral Equations}, 6(5):1025--1031, 1993.

\bibitem{Lou}
Y.~Lou.
\newblock On the effects of migration and spatial heterogeneity on single and
  multiple species.
\newblock {\em J. Differential Equations}, 223(2):400--426, 2006.

\bibitem{MR06}
W.~J. Ni and M.~X. Wang.
\newblock Long time behavior of a diffusive competition model.
\newblock {\em Appl. Math. Lett.}, 58:145--151, 2016.

\end{thebibliography}

\end{document}